\documentclass[11pt]{article}
\usepackage[justification=centering]{caption}
\usepackage{setspace}
\usepackage[rmargin=0.35in,lmargin=0.35in,tmargin=0.35in,bmargin=0.5in,paperwidth=5.5in,paperheight=8.5in]{geometry}

\usepackage{amsmath, amsfonts, amssymb,amsthm}
\usepackage{mathtools}
\usepackage{stackengine}
\usepackage{fancyhdr}

\usepackage[utf8]{inputenc}
\mathtoolsset{showonlyrefs}

\usepackage{enumitem}

\usepackage[table]{xcolor}
\usepackage[colorinlistoftodos]{todonotes}

\usepackage{hyperref}
\usepackage{tikz-cd}
\usepackage{tikz}

\usepackage[utf8]{inputenc}

\usepackage[
firstinits=true,
backend=biber,
style=alphabetic,
sorting=nyt,
doi=false,
isbn=false,
url=false,
eprint=true,
]{biblatex}
\usepackage{cleveref}

\usepackage{graphicx}

\usepackage{bbm}
\usepackage{bm}

\usepackage[skip=10pt]{caption}
\usepackage{subcaption}
\usepackage{float}
\usepackage{bbm}
\usepackage{bm}
\usepackage{array}

\usepackage{flafter}

\usepackage{ifthen}

\usepackage{multirow}

\usepackage{comment}

\usepackage{etoolbox}
\tracingpatches
\makeatletter
\patchcmd{\@setref}{??}{\bfseries\color{red} [Label]}{}{}
\makeatother

\makeatletter
\def\abx@missing@entry#1{%
	{\bfseries\color{red} #1}}
\makeatother

\makeatletter
\patchcmd{\@@setcref}         {??}{\color{red} [\filename{#2}]}{}{}
\patchcmd{\@@setcref}         {??}{\color{red} [\filename{#2}]}{}{}
\patchcmd{\@@setcrefrange}    {??}{\color{red} [\filename{#2}]}{}{}
\patchcmd{\@@setcrefrange}    {??}{\color{red} [\filename{#2}]}{}{}
\patchcmd{\@@setcrefrange}    {??}{\color{red} [\filename{#2}]}{}{}
\patchcmd{\@@setcrefrange}    {??}{\color{red} [\filename{#2}]}{}{}
\patchcmd{\@@setcrefrange}    {??}{\color{red} [\filename{#2}]}{}{}
\patchcmd{\@@setcrefrange}    {??}{\color{red} [\filename{#2}]}{}{}
\patchcmd{\@@setnamecref}     {??}{\color{red} [\filename{#2}]}{}{}
\patchcmd{\@@setnamecref}     {??}{\color{red} [\filename{#2}]}{}{}
\patchcmd{\@@setcpageref}     {??}{\color{red} [\filename{#2}]}{}{}
\patchcmd{\@@setcpageref}     {??}{\color{red} [\filename{#2}]}{}{}
\patchcmd{\@@setcpagerefrange}{??}{\color{red} [\filename{#2}]}{}{}
\patchcmd{\@@setcpagerefrange}{??}{\color{red} [\filename{#2}]}{}{}
\patchcmd{\@@setcpagerefrange}{??}{\color{red} [\filename{#2}]}{}{}
\patchcmd{\@@setcpagerefrange}{??}{\color{red} [\filename{#2}]}{}{}
\patchcmd{\@@setcpagerefrange}{??}{\color{red} [\filename{#2}]}{}{}
\patchcmd{\@@cref}            {??}{\color{red} [\filename{#2}]}{}{}
\makeatother

\makeatletter
\newcommand{\replunderscores}[1]{\expandafter\@repl@underscores#1_\relax}

\def\@repl@underscores#1_#2\relax{%
	\ifx \relax #2\relax
	#1%
	\else
	#1%
	\textunderscore
	\@repl@underscores#2\relax
	\fi
}
\makeatother

\newcommand{\filename}[1]{\emph{\replunderscores{#1}}}

\usepackage{xparse}

\hypersetup{colorlinks=true,
	linkcolor=black,
	filecolor=magenta,      
	urlcolor=cyan,
	citecolor=teal}

\newtheorem{defi}{Definition}[section]
\theoremstyle{definition}
\newtheorem{theorem}[defi]{Theorem}
\newtheorem*{theorem*}{Theorem}
\newtheorem{lemma}[defi]{Lemma}
\newtheorem{conj}[defi]{Conjecture}
\newtheorem*{conj*}{Conjecture}
\newtheorem{coro}[defi]{Corollary}
\newtheorem{prop}[defi]{Proposition}
\newtheorem*{claim}{Claim}

\newtheorem{examples}[defi]{Examples}

\newtheorem*{notation}{Notation}
\newtheorem*{note1}{Note}
\newtheorem*{note2}{Notes}

\newtheoremstyle{conv}{}{}{\bfseries\centering}{}{\itshape\bfseries}{:}{ }{}

\theoremstyle{conv}

\renewcommand{\epsilon}{\varepsilon}
\renewcommand{\phi}{\varphi}

\renewcommand{\bar}{\overline}
\newcommand{\1}{\mathbbm{1}}

\newcommand{\R}{\mathbb{R}}
\newcommand{\Z}{\mathbb{Z}}
\newcommand{\N}{\mathbb{N}}

\newcommand{\B}{\mathbb{B}}
\renewcommand{\S}{\mathbb{S}}

\newcommand{\C}{\mathcal{C}}

\newcommand{\Ggraph}[2]{\mathcal{G}_{#1}\left(#2\right)}
\newcommand{\Ne}[1]{\mathcal{N}\left(#1\right)}
\newcommand{\conn}[1]{\text{conn}\left(#1\right)}

\NewDocumentCommand{\Bor}{O{\epsilon} O{d}}{\text{Bor}^{#2}\left(#1\right)}

\usepackage{ifthen}

\NewDocumentCommand{\distgraph}{O{\epsilon} O{\alpha} O{d}}%
	{\ifthenelse{\equal{#1}{0}}{\S^{#3}_{#2}}{	\S^{#3}_{#2}\left(#1\right)}	
	}

\NewDocumentCommand{\gendistgraph}{O{X} O{\alpha}}{\text{G}_{#2}\left(#1\right)}

\renewcommand{\P}[1]{\mathbb{P}\left[#1\right]}

\newcommand{\dist}[2][]{%
	\ifthenelse{ \equal{#1}{} }
	{\mathbbm{d}\left(#2\right)}
	{\mathbbm{d}_{#1}\!\left(#2\right)}
}

\newcommand{\diam}[1]{\text{diam}\left(#1\right)}

\newcommand{\interior}[1]{\text{int}\left(#1\right)}

\newcommand{\Del}[1]{\mathcal{D}\!\text{\textit{el}}\!\left(#1\right)}

\renewcommand{\cref}{\Cref}

\title{\textbf{\MakeUppercase{Topology and chromatic number of random $\epsilon$-distance graphs on spheres}}}

\author{\bf Francisco Martinez-Figueroa \\
The Ohio State University\\
e-mail address: martinezfigueroa.2@osu.edu}
\date{May 2022}

\begin{document}

\maketitle

\section{Introduction}

Given a metric space $(Z,\mathbbm{d})$ and a parameter $\alpha\geq0$, the corresponding \textit{distance graph} on $Z$ is obtained by drawing edges $x\sim y$ whenever $x,y\in Z$ and $\dist{x,y}=\alpha$. Distance graphs have been an active topic of research in the last few decades, popularized by their relation to Erd\H{o}s' \textit{$n$ distinct distances} problem \cite{Erdos1946} and \textit{unit distances} problem \cite[see also][Ch.~4]{Matousek2002}, among others. Of particular interest, has been the study of the chromatic number of unit distance graphs, that is, when $\alpha=1$, on Euclidean spaces. For instance, the famous \textit{Hardwiger--Nelson}'s problem asks to compute the chromatic number of the unit distance graph on the plane $\R^2$. Similar distance graphs have also been studied in higher dimensions \cite{KupavskiiRaigorodskii2009,Raigorodskii2018,KahleBirra2015,Coulson2002}, on spheres \cite{Simmons1976,Kupavskii2011,Raigorodskii2012,Prosanov-Raigorodskii-Sagdeev2017,Lovasz1983}, and even on hyperbolic spaces \cite{DeCorte2019,Parlier2016} (see Soifer's  \textit{Mathematical Coloring Book} \cite{Soifer2014} for a detailed historical revision).

$\epsilon$-Distance graphs are a natural generalization obtained when relaxing the condition on the edges: instead of making $x\sim y$ only when $\dist[]{x,y}=\alpha$, we may draw an edge $x\sim y$ whenever their distance is $\epsilon$-\textit{close} to $\alpha$. This kind of graphs and their chromatic number, have also attracted researchers in the past few years: Exoo \cite{Exoo2005}, Bock \cite{Bock2019} and, Currie--Eggleton \cite{Currie2015}, among others, have studied the chromatic number of unit $\epsilon$-distance graphs on the plane, for $\epsilon$ small. 

Provided $Z$ is also endowed with a probability measure, we may define \textit{random $\epsilon$-distance graphs}: given $n\in\N$, take $n$ i.i.d. random points on $Z$ and consider the induced subgraph on the $\epsilon$-distance graph. In the case when $Z=\S^d$, endowed with the uniform Borel probability measure, and geodesic distance, varying the values of $\alpha$ provides a generalization of the random Borsuk graphs as studied in \cite{Kahle-Martinez2020}\footnote{Notice in \cite{Kahle-Martinez2020} we used the Euclidean distance on the ambient space, however the geodesic distance is a better choice in this setting}.

In this paper, we are interested in the chromatic number of random $\epsilon$-distance graphs on spheres (with the uniform probability measure and geodesic distance) when $n\to\infty$ and $\epsilon=\epsilon(n)\to0$. In \cite{Kahle-Martinez2020}, we established that topological tools (namely Borsuk--Ulam's and Lyusternik--Schnirerlman's theorems) can effectively produce lower bounds for the chromatic number of random Borsuk graphs. Here, we explore the effectiveness of topological invariants for this general case. Specifically, we study the connectivity of the \textit{neighborhood complex}, introduced by L\'ovasz in his breakthrough proof of Kneser's conjecture \cite{Lovasz1978}, which, as an application of Borsuk--Ulam's Theorem, gives a general lower bound for the chromatic number of graphs.  

Our main result is that, in general, the connectivity of the neighborhood complex (\cref{thm:conn-neigh-Lovasz}) can at most produce a lower bound of $d+2$, while standard results for distance graphs on spheres give an exponential bound. Thus, in general, this topological invariant is far away from the chromatic number. However, for dimensions $d=1$ and 2, we show this invariant is a.a.s. $d+2$, and it is a tight bound for a range of values $\alpha$. We conjecture this to be also the case for dimensions $d\geq3$, provided $\alpha$ is close to $\pi$. We point out that we don't consider the case $\alpha=\pi/2$, so whether L\'ovasz's bound can be tight in that case is still an open question.

\section{Preliminaries}
\subsection{Definitions}

We formally write down the definitions from the introduction. 

\begin{defi}[Distance Graph]\label{def:distance_graphs}
	Given a metric space $(Z,\mathbbm{d})$, a subset $X\subset Z$, and a constant $\alpha$, the corresponding \textbf{distance graph} $\gendistgraph$, is the graph given by
	\begin{itemize}
		\item \textbf{Vertices:} $V(\gendistgraph)=X$, and
		\item \textbf{Edges:} $x\sim y$ in $\gendistgraph$ whenever $\dist{x,y}=\alpha$, and $x\neq y$. 
	\end{itemize}
\end{defi}

\begin{defi}[$\epsilon$-Distance Graph]\label{def:eps_distance_graph}
	Given a metric space $(Z,\mathbbm{d})$, a subset $X\subset Z$, a constant $\alpha\geq0$, and a parameter $0\leq\epsilon\ll\alpha$. The corresponding \textbf{$\bm{\epsilon}$-distance graph} $\gendistgraph[X,\epsilon]$, is the graph given by
	\begin{itemize}
		\item \textbf{Vertices:} $V(\gendistgraph[X,\epsilon])=X$, and
		\item \textbf{Edges:} $x\sim y$ in $\gendistgraph[X,\epsilon]$ whenever $\alpha-\epsilon\leq\dist{x,y}\leq\alpha+\epsilon$, and $x\neq y$. 
	\end{itemize}
\end{defi}

Note that $\gendistgraph=\gendistgraph[X,0]$.

\begin{defi}[Random $\epsilon$-Distance Graph]
	Given a measure metric space $(Z,\mathbbm{d},\mu)$ with probability measure $\mu$, a natural number $n$, and constants $\alpha, \epsilon\geq 0$. The corresponding \textbf{random $\bm{\epsilon}$-distance graph} $\gendistgraph[n,\epsilon]$ is the induced subgraph of $\gendistgraph[Z,\epsilon]$ with vertices $X=\{X_1,\dots, X_n\}\subset Z$ where the points $X_i$ are drawn i.i.d. from $Z$ according to $\mu$. 
\end{defi}

We focus on random $\epsilon$-distance graphs on the unit sphere $\S^d$, with $\mathbbm{d}$ the geodesic distance and $\mu$ the uniform probability measure. Thus, we introduce the following notation.

\begin{notation}
    Let $\distgraph:=\gendistgraph[\S^d,\epsilon]$ and $\distgraph[n,\epsilon]$, be the random $\epsilon$-distance graph on $\S^d$. For a set $X\subset\S^d$, we write $\distgraph[X,\epsilon]$ for the induced subgraph of $\distgraph[\epsilon]$ with vertex set $X$. For distance graphs with $\epsilon=0$, we simply write $\distgraph[0]=\distgraph[0](0)$
\end{notation}

\begin{note1} 
Other authors use the notation $\S^d_r$ for the sphere of radius $r$, and look at unit distance graphs on it. To be consistent with them, we would need to use the notation $\distgraph[\epsilon][1/\alpha]$ and Euclidean distance instead of geodesic, but we avoid it for being too cumbersome.
\end{note1}

We point out that the graphs $\distgraph[n,\epsilon]$ are indeed a generalization of random Borsuk graphs \cite{Kahle-Martinez2020}, setting $\alpha=\pi$: $$\Bor[n,\epsilon]=\distgraph[n,\epsilon][\pi].$$

On the opposite side, setting $\alpha=0$, produces a \textit{random geometric graph} on the sphere.  While we don't study this case, it provides a rich family of random graphs where the geometry of the underlying space can be used to derive graph theoretic properties. We refer the reader to Penrose's book \cite{Penrose2003} for an in-depth exposition. 

L\'ovasz introduced the following simplicial complex in his proof to Kneser's conjecture \cite{Lovasz1978}.

\begin{defi}Given a graph $G$, its \textbf{neighborhood complex} is the simplicial complex $\Ne{G}$ with vertices the vertices of $G$ and simplices all $A\subset V(G)$ such that $\mathop{\cap}_{v\in A}N(v)\neq\emptyset$. Where $N(v)$ is the set of neighbors of $v$ in the graph $G$.  
\end{defi} 

The great breakthrough of Lov\'asz's paper is a relationship between the connectivity of $\Ne{G}$ and $\chi(G)$ (the chromatic number of $G$). It is worth noting that the topological tool behind this relationship is Borsuk--Ulam's Theorem. 

\begin{defi}[Connectivity]
	Let $X$ be a non-empty topological space. For an integer $k\geq 0$, we say $X$ is \textbf{$\bm{k}$-connected} if for every $\ell=0, 1,\dots, k$, any continuous map $f:\S^\ell\to X$ can be extended to a continuous map $\B^{\ell+1}\to X$. We say $X$ is \textbf{connected} if it is $0$-connected. We say $X$ is \textbf{simply connected} if it is 1-connected.
\end{defi}

We denote by $\conn{X}$ the largest $k$ such that $X$ is $k$-connected. It is a classical result that $\conn{\S^d}=d-1$ for all $d\geq 0$.

\begin{theorem}[\cite{Lovasz1978}]\label{thm:conn-neigh-Lovasz}
	Let $G$ be a graph. Then $$\conn{\Ne{G}}+3\leq\chi(G).$$
\end{theorem}

We are interested in the effectiveness of this topological bound for the chromatic number $\chi(\distgraph[n,\epsilon])$ when $n\to\infty$ and $\epsilon=\epsilon(n)\to 0$. 

Since we examine the efficiency of topological methods, we will only consider the case when the vertices of the graph form an $(\epsilon/4)$-net of $\S^d$. Adapting \cite[Lemma 3.3]{Kahle-Martinez2020} to geodesic distance, we get a threshold for $\epsilon$ to get nets with high probability (\cref{lemma:eps_net_regime}). Here, and throughout this paper, we say that an event $A$ happens \textit{asymptotically almost surely (a.a.s.)} if $\P{A}\to1$ as $n\to\infty$.  

\begin{lemma}\label{lemma:eps_net_regime}
	Given a dimension $d\geq0$, there exists a constant $\C_d$ such that if $\epsilon\geq \C_d\left(\dfrac{\log n}{n}\right)^{1/d}$ and $\epsilon\to 0$, then a.a.s. the vertices of $\distgraph[n,\epsilon]$ form an $(\epsilon/4)$-net of $\S^d$. 
\end{lemma}

\subsection{Some Easy Bounds}

We start with some easy bounds for the chromatic number of random $\epsilon$-distance graphs on spheres, and state some evident relations with the well studied cases of distance graphs on spheres and $\R^d$. 

We will need the following result. 
\begin{prop}\label{lemma:net_connected}
	Let $X\subset\S^d$ be a finite $(\epsilon/2)$-net. Let $H$ be the graph with vertex set $X$ and edges $x\sim y$ in $H$ whenever $\dist{x,y}\leq\epsilon$. Then $H$ is connected. 
\end{prop}

\begin{proof}
	Fix a vertex $x_0\in X$ and let $C_0$ be the connected component of $H$ containing $x_0$. Suppose by way of contradiction that $H$ is not connected, so $X\setminus C_0\neq\emptyset$. Choose $x\in C_0$ and $y\in X\setminus C_0$ such that they minimize $\dist{x,y}$. Since they are in different connected components, $\dist{x,y}>\epsilon$. Write $\dist{x,y}=\epsilon+4\delta$. Let $[x,y]$ denote the circular arc between $x$ and $y$, and let $a,b\in[x,y]$ be the points such that $\dist{x,a}=\dist{y,b}=\epsilon/2+\delta$, so $\dist{y,a}=\dist{x,b}=\epsilon/2+3\delta$. Since $X$ is an $(\epsilon/2)$-net, there are points $x',y'\in X$ such that $\dist{x',a}\leq\epsilon/2$ and $\dist{y',b}\leq\epsilon/2$. Since 
	\begin{align*}
		\dist{x,x'}& \leq\dist{x,a}+\dist{a,x'}\leq\epsilon+\delta<\dist{x,y}\text{ and}\\
		\dist{y,y'}&\leq\dist{y,b}+\dist{b,y'}\leq\epsilon+\delta<\dist{x,y},
	\end{align*} 
	it must be that $x'\in C_0$ and $y'\in X\setminus C_0$. But then
	$$\dist{x',y'}\leq\dist{x',a}+\dist{a,b}+\dist{b,y'}\leq\epsilon/2+2\delta+\epsilon/2=\epsilon+2\delta<\dist{x,y},$$
	which is contradiction. Therefore $C_0=X$ as desired. 
\end{proof}

\begin{lemma}\label{lemma:clique_lower_bound}
	Suppose there is a $k$-clique in $\distgraph[0]$. If   $\epsilon\geq \C_d\left(\log n/n\right)^{1/d}$ and $\epsilon\to 0$, then a.a.s. $k+1\leq\chi\left(\distgraph[n,\epsilon]\right)$.
\end{lemma}

\begin{proof}
	By \cref{lemma:eps_net_regime} suppose $X=V(\distgraph[n,\epsilon])$ is an $(\epsilon/4)$-net of $\S^d$, and let $\{v_1,\dots, v_k\}$ be a $k$-clique of $\distgraph[0]$. For any $x\in \S^d$, because of its rotational symmetry, we can always find a rotation $\Phi$ such that $\Phi(v_1)=x$, and thus $\{x, \Phi(v_2),\dots, \Phi(v_k)\}$ is also a $k$-clique of $\distgraph[0]$. In other words, we can find a $k$-clique containing any fixed vertex of $\S^d$. 
	
	By way of contradiction, suppose $G=\distgraph[X,\epsilon]$ is $k$-colorable, and fix such a coloring. Let $x,y\in X$ such that $\dist{x,y}\leq\epsilon/2$, we will prove $c(y)=c(x)$. Let $\{x'_1=x,x'_2,x'_3,\dots, x'_k\}$ be a clique of $\distgraph[0]$, so $\dist{x_i',x_j'}=\alpha$ whenever $i\neq j$. For each $2\leq i\leq k$ there exists a point $x_i\in X$ with $d(x_i,x'_i)\leq\epsilon/4$ and let $x_1=x'_1=x$. Thus 
	\begin{align*}
		\dist{x_i,x_j}&\leq \dist{x_i,x_i'}+\dist{x_i',x_j'}+\dist{x_j',x_j}\leq \epsilon/2+\alpha+\epsilon/2=\alpha+\epsilon\\
		\dist{x_i,x_j}&\geq \dist{x_i',x_j'}-\dist{x_i,x_i'}-\dist{x_j,x_j'}\geq \alpha-\epsilon/2-\epsilon/2=\alpha-\epsilon
	\end{align*}
	Thus $x_i\sim x_j$ in $G$ for all $i\neq j$, and so $\{x_1,\dots, x_k\}$ is a $k$-clique of $G$. That means the coloring $c$ must assign a different color to each $x_i$, say $c(x_i)=i$. Now, for $y$ and each $x_i$ we get 
	\begin{align*}
		\dist{y,x_i}&\leq\dist{y,x}+\dist{x,x_i'}+\dist{x_i',x_i}\leq \epsilon/2 +\alpha+\epsilon/4\leq\alpha+\epsilon\text{, and}\\
		\dist{y,x_i}&\geq\dist{x,x_i'}-\dist{x,y}-\dist{x_i,x_i'}\geq \alpha-\epsilon/2-\epsilon/4\geq\alpha-\epsilon
	\end{align*}
	So $y\sim x_i$ for all $i\geq2$, and thus $\{y,x_2, \dots, x_k\}$ is also a $k$-clique in $G$, in particular, this forces $c(y)=1=c(x)$, as desired. 
	
	To finish the proof, consider the graph $H$ with vertex set $X$ and edges $x\sim y$ in $H$, whenever $\dist{x,y}\leq\epsilon/2$. By our work above, if $x$ and $y$ are on the same connected component of $H$, then $c(x)=c(y)$. However, applying \cref{lemma:net_connected}, $H$ is a connected graph, so $c$ must be constant. But since $x\sim x_2$ in $G$, $c(x)\neq c(x_2)$ which is a contradiction. Therefore $k+1\leq\chi(G)$. 
\end{proof}


To apply this lemma, we need cliques on $\distgraph[0]$, which can be obtained with inscribed regular simplices, so the following notation is useful. 

\begin{notation}
	Let $\ell_d$ denote the arclength between two vertices of the $(d+1)$-dimensional regular simplex inscribed in $\S^d$. Note we have the inequalitites $\ell_1<\ell_2<\ell_3<\dots<\pi$. 
\end{notation}

\begin{coro}\label{coro:clique_bound_eps_distance_graphs}
	 If $\epsilon\geq \C_d(\log n/n)^{1/d}$, and $\epsilon\to 0$, then the following hold:
	\begin{enumerate}
		\item If $\alpha=\ell_d$, then a.a.s. $d+3\leq\chi\left(\distgraph[n,\epsilon][\ell_d]\right)$.
		\item If $0< \alpha\leq \ell_{d-1}$, then a.a.s. $d+2\leq\chi\left(\distgraph[n,\epsilon]\right)$.
	\end{enumerate}
\end{coro}

\begin{proof}
	For (1), note the distance graph $\distgraph[0][\ell_d]$ has a $(d+2)$-clique, formed by the vertices of a regular inscribed $(d+1)$-simplex. So the result follows from applying the previous the lemma. 
	
	We now prove (2). As above, the vertices of the regular $d$-simplex form a $(d+1)$-clique in $\distgraph[0][\ell_{d-1}][d-1]$. If $0<\alpha\leq\ell_{d-1}$, then contracting $\S^d$ by a factor of $\frac{\alpha}{\ell_{d-1}}\leq 1$, produces a $(d-1)$-sphere of radius less than 1, where the arclenght between the vertices of the $d$-simplex is exactly $\alpha$. It's straightforward to check that this $(d-1)$-sphere can be isometrically embedded into the unit sphere $\S^d$ (by intersecting a $d$-plane with $\S^d$ at the necessary height). So the vertices of the $d$-simplex inscribed in $\frac{\alpha}{\ell_{d-1}}\S^{d-1}$ now form a $(d+1)$-clique on $\distgraph[0]$. Thus, applying the previous lemma, $d+2\leq\chi(\distgraph[n,\epsilon])$. 
\end{proof}

This lemma and corollary already exhibit some differences between $\epsilon$-distance graphs and usual distance graphs. Having a clique in a distance graph, will only give the trivial bound $\omega(G)\leq\chi(G)$, but for the $\epsilon$-distance graphs, we get a small improvement. As we will see in \cref{sec:conn_nei_complex_eps_distance_graphs}, the connectivity of the neighborhood complex cannot give better bounds than those in \cref{coro:clique_bound_eps_distance_graphs}, for $\alpha\neq\pi/2, \alpha\leq\ell_d$. What's even more, we will see in section \cref{sec:case_1d} that \cref{coro:clique_bound_eps_distance_graphs} already gives tight bounds for all $\distgraph[n,\epsilon][\alpha][1]$ in dimension $d=1$, for $0<\alpha\leq\ell_1$. 

In general, bounds for the chromatic number of distance graphs on spheres, also provide bounds for random $\epsilon$-distance graphs. The following lemma formalizes some of the most evident relations. 

\begin{lemma}\label{lemma:rel_distance_and_eps_distance_graphs}
	Let $X\subset\S^d$ be a finite $(\epsilon/2)$-net of $\S^d$, for $\epsilon$ sufficiently small.  
	\begin{enumerate}
		\item If $H$ is a finite subgraph of $\distgraph[0]$, then $\chi(H)\leq\chi(\distgraph[X,\epsilon])$.
		\item If there is a $k$-coloring for $\distgraph[0]$ with color classes $C_1, \dots, C_k$ such that for each $i$, we can partition $C_i=A_1^{(i)}\cup\cdots\cup A_{r_i}^{(i)}$, satisfying
		\begin{itemize}
			\item $\diam{A_j^{(i)}}<\alpha$ for all $j=1, \dots, r_i$, and
			\item $\dist{\bar{A_j^{(i)}},\bar{A_k^{(i)}}}>\alpha$, whenever $j\neq k$. 
		\end{itemize}
	then $\chi(\distgraph[X,\epsilon])\leq\chi(\distgraph)\leq k$. 
	\end{enumerate}
\end{lemma}

\begin{proof}\leavevmode
	\begin{enumerate}
		\item Suppose $V(H)=\{y_1,\dots, y_k\}$. Since $X$ is an $(\epsilon/2)$-net, there are points $x_1, \dots, x_k\in X$ such that $\dist{x_i, y_i}\leq\epsilon/2$. We claim that the map $f:H\to \distgraph[X,\epsilon]$ that assigns $f(y_i)=x_i$ is a graph homomorphism. Thus, since $H\not\to K_{\chi(H)-1}$, we must also have $\distgraph[X,\epsilon]\not\to K_{\chi(H)-1}$, and so $\chi(H)\leq \chi(\distgraph[X,\epsilon])$. To prove the claim, suppose $y_i\sim y_j$ in $H$, so $\dist{y_i,y_j}=\alpha$, since $H\subset\distgraph[0]$. Then \begin{align*}
			\dist{x_i,x_j}&\leq\dist{x_i,y_i}+\dist{y_i,y_j}+\dist{y_j,x_j}\leq\alpha+\epsilon\text{ and,}\\
			\dist{x_i,x_j}&\geq \dist{y_i,y_j}-\dist{x_i,y_i}-\dist{y_j,x_j}\geq\alpha-\epsilon,
		\end{align*}
	so $x_i\sim x_j$ in $\distgraph[X,\epsilon]$ as desired. 
	
	\item Since there is a finite number of sets $A_j^{(i)}$, we can find the smallest $\epsilon$ such that $\diam{A_j^{(i)}}<\alpha-\epsilon$, for all $i$ and $j$, and such that $\dist{\bar{A_j^{(i)}},\bar{A-k^{(i)}}}>\alpha+\epsilon$ for all $j\neq k$, $i=1, \dots, k$. Then, clearly all edges in $\distgraph[\epsilon]$ must have both vertices on the same set $A_j^{(i)}$, so the color classes $C_1, \dots, C_k$ provide a proper $k$-coloring for $\distgraph[\epsilon]$ as well. \qedhere
	\end{enumerate}
\end{proof}

By the De Bruijn--Erd\H{o}s Theorem \cite[see][p.~215]{Diestel2017}, for every $d$ and $\alpha$, there is some finite subgraph $H\subset\distgraph[0]$ such that $\chi(H)=\chi(\distgraph[0])$, so part 1 of \cref{lemma:rel_distance_and_eps_distance_graphs} implies that a.a.s. $\chi(\distgraph[0])\leq\chi\left(\distgraph[n,\epsilon]\right)$, provided $\epsilon$ is as in \cref{lemma:eps_net_regime}. Note \cite[Lemma~1]{Kostina2019} states that for each fixed $0<\alpha<\pi$ the following bound holds:
$$(c(\alpha)+o(1))^d\leq\chi(\distgraph[0]),$$
where $c(\alpha)$ is a constant depending only on $\alpha$. Then, a.a.s. we get $$(c(\alpha)+o(1))^d\leq\chi(\distgraph[n,\epsilon])$$
for $0<\alpha<\pi$, provided $\epsilon$ is as in \cref{lemma:eps_net_regime}. 

On the other hand, while part 2  may seem too special, most constructions of upper bounds for spheres do satisfy these requirements, since they arise from tessellating the space with some cells $A_j^{(i)}$, and coloring them so that cells assigned the same color are far away. Thus, in general, the same known upper bounds for $\distgraph[0]$, which are exponential on $d$, are also upper bounds for $\distgraph[n,\epsilon]$ a.a.s. See \cite{Soifer2014,Prosanov2018,Raigorodskii2012,Lovasz1983,Kupavskii2011,Simmons1976,Coulson2002} for a complete overview. 

The cases $\alpha=0$ and $\alpha=\pi$, however, behave differently. The case $\alpha=\pi$ gives a random Borsuk graph, so it's chromatic number is a.a.s. $d+2$. The case $\alpha=0$, gives a random geometric graph on $\S^d$, for which it can be seen the chromatic number increases as $\epsilon\to 0$ and $n\to\infty$ \cite[see][]{Penrose2003} (in the regime from \cref{lemma:eps_net_regime}).

\section{Connectivity of Neighborhood Complex}\label{sec:conn_nei_complex_eps_distance_graphs}
In this section, we bound the connectivity of the neighborhood complex of $\distgraph[n,\epsilon]$, thus limiting how close this topological invariant can get to the chromatic number. 

\begin{theorem}\label{thm:best_conn_neigh_eps_distance_graphs}
	Fix a dimension $d\geq1$ and $0\leq\alpha\leq\pi$, $\alpha\neq\pi/2$.  If $\epsilon\geq \C_d(\log n/n)^{1/d}$ and $\epsilon\to0$, then a.a.s.  there is an injection $$\Z\hookrightarrow H_d\left(\Ne{\distgraph[n,\epsilon]}\right),$$ in particular, $\conn{\Ne{\distgraph[n,\epsilon]}}\leq d-1$.
\end{theorem}

Note, however, that \cref{thm:best_conn_neigh_eps_distance_graphs} excludes the case $\alpha=\pi/2$, so there is still hope that topological methods could give efficient lower bounds for this particular case.

The proof of this theorem uses Delaunay Triangulations of the sphere, and depends strongly on the fact that the  geodesic Delaunay Triangulation is ambiently isotopic to $\S^d$, for a dense enough $\epsilon$-net. While we won't prove this, we give the necessary definition and results, following closely \cite{Leibon2000} and \cite{EdelShah97}. We refer the reader to \cite{EdelShah97,Leibon2000, Boissonnat2017} for all the details, and to \cite{Voronoi1908,Edelsbrunner1987,Delaunay1934} for more context. It is worth mentioning that \cite{Boissonnat2017} later disproved that the conditions described in \cite{Leibon2000} are sufficient to get geodesic Delaunay triangulations for an arbitrary Riemmannian Manifold, however, they do work for certain manifolds, including those with constant curvature, such as spheres \cite[p.~1]{Boissonnat2017}. 

\begin{defi}[Delaunay Triangulation]
	Given a finite set $X\subset\R^d$ in general position. It's \textbf{Delaunay Triangulation} is the simplicial complex $\Del{X}$ with simplices all subsets $\sigma\subset X$ for which there exists a closed ball $B$ such that $X\cap B=\sigma$ and $X\cap\interior{B}=\emptyset$. 
\end{defi}

Note that $\Del{X}$ has a natural geometric realization, namely, each $\sigma\in\Del{X}$, is the convex hull of its vertices in $\R^d$. Since $X$ is in general position, each $\sigma$ is guaranteed to be a geometric simplex.  

The same definition makes sense on a Riemannian manifold $M$, where \textbf{geodesic closed ball} means a set of the form $B(x,r)=\{y\in M: \dist[M]{x,y}\leq r\}$. 

\begin{defi}[Geodesic Delaunay]
	Let $M$ be a Riemannian manifold and $X\subset M$ be a finite subset. It's \textbf{geodesic Delaunay Triangulation} is the simplicial complex with simplices all subsets $\sigma\subset X$ for which there exists a geodesic closed ball $B\subset M$, such that $X\cap B=\sigma$ and $X\cap\interior{B}=\emptyset$. 
\end{defi}
While this certainly produces a simplicial complex, it is in general not unique, and it may not satisfy the same properties as the Euclidean Delaunay Triangulations. Among other things, the points $X$ should be \textit{generic} and \textit{dense} enough so that geodesics can be approximated by line segments. In \cite{Leibon2000}, they provide precise definitions of what this means in general, however, for the case $M=\S^d$, it is easy to check that these are trivially satisfied by any random $(\epsilon/4)$-net (with probability 1), provided $\epsilon$ is small. 

In \cite{Leibon2000}, they also study properties to ensure the geodesic Delaunay Triangulation is isomorphic to the Euclidean one. Their proofs involve restricted Voronoi Diagrams (as defined by \cite{EdelShah97}) and the duality between Delaunay Triangulations and Voronoi Diagrams. Instead, to keep our exposition short, we circumvent the introduction of Voronoi Diagrams by giving an explicit argument for spheres.  Note that any small geodesic ball on $\S^d$ is the intersection of $\S^d$ with a small Euclidean ball with the same center. Suppose $X\subset\S^d$ is a generic $(\epsilon/4)$-net, with $\epsilon$ small. Then, if $\sigma\in\Del{X}$ there is an Euclidean ball $B$ such that $X\cap B=\sigma$ and $X\cap\interior{B}=\emptyset$, taking $B'=B\cap\S^d$ is a geodesic ball that makes $\sigma$ a simplex in the geodesic Delaunay Triangulation. Similarly, if $\sigma$ is in the geodesic Delaunay, and $B'$ is a geodesic ball satisfying the definition, then taking the Euclidean ball with same center and Euclidean radius, makes $\sigma\in\Del{X}$.  

The following theorem from \cite{Leibon2000} guarantees $\Del{X}$ is homeomorphic to $\S^d$. While they provide a general result for submanifolds in $\R^d$, for spheres it simplifies to:

\begin{theorem}[{\cite[Thm.~5.3]{Leibon2000} for spheres}]\label{thm:delaunay_homeomorphic_sphere}
	If $X\subset \S^d$ is a generic, dense enough finite subset, then $\Del{X}$ is ambiently isotopic to $\S^d$. In particular, there is a homeomorphism $\Del{X}\approx \S^d$. 
\end{theorem}

In \cite[p.~344]{Leibon2000}, they describe this ambient isotopy $\Psi:\Del{X}\to\S^d$. For $x\in\Del{X}$, $\Psi(x)$ is the point on $\S^d$ closest to $x$ (in Euclidean distance). It is clear that for each point $x\in\R^{d+1}\setminus\{0\}$, the closest point of $x$ on $\S^d$ is the projection $\frac{x}{\|x\|}$, so we get the following easy but useful lemma. 

\begin{lemma}\label{lemma:ambien_homeo_delaunay}
	Let $X\subset\S^d$ be a generic, dense enough finite net so that $\Del{X}\approx\S^d$ as described above. Then, the homotopy equivalence $\Psi:\R^{d+1}\setminus\{0\}\to\S^d$ given by $\Psi(x)=\frac{x}{\|x\|}$ becomes a homeomorphism when restricted to $\Del{X}$. 
\end{lemma}

\subsection{Some auxiliary results}
We establish some intermediate results needed for \cref{thm:best_conn_neigh_eps_distance_graphs}. The following lemma bounds the diameter of the simplices in a Delaunay Triangulation.
   
\begin{lemma}\label{lemma:diam_simplex_delaunay}
	Let $X\subset\S^d$ be a finite $(\epsilon/4)$-net, with $\epsilon$ small enough so that $\Del{X}\approx\S^d$. Then for any simplex $\sigma=[v_0,\dots,v_k]\in\Del{X}$, we have $\dist[\S^d]{v_i,v_j}\leq\epsilon/2$ for all $0\leq i,j\leq k$. 
\end{lemma}

\begin{proof}
	Since $\sigma$ is in the Delaunay Triangulation, it is in the geodesic Delaunay, and there must be a geodesic ball $B(x,r)$ such that $X\cap B(x,r)=\sigma$ and $X\cap\interior{B(x,r)}=\emptyset$, for some $x\in\S^d$ and $r>0$. Thus $\dist{v_i,x}=r$ for $0\leq i\leq k$, and any other point of $X$ is at at a distance larger than $r$ from $x$. Since $X$ is an $(\epsilon/4)$-net it follows $\dist{v_i,x}\leq\epsilon/4$, and so $\dist{v_i,v_j}\leq \dist{v_i,x}+\dist{x,v_j}\leq\epsilon/2$.
\end{proof}

The following Theorem gives the relationship between the Delaunay Triangulation and the Neighborhood complex.
\begin{theorem}\label{thm:delaunay_subcomplex_neighbor}
	Fix a dimension $d$ and $0\leq\alpha\leq\pi$. Suppose $\epsilon\geq \C_d(\log n/n)^{1/d}$ and $\epsilon\to 0$.   Let $X$ denote the (random) vertex set of $\distgraph[n,\epsilon]$. Then a.a.s. $\Del{X}\subset\Ne{\distgraph[n,\epsilon]}$, as a sub-complex. 
\end{theorem}

\begin{proof}
	From \cref{thm:delaunay_homeomorphic_sphere}, a.a.s. $\Del{X}$ is a simplicial complex homeomorphic to $\S^d$, with all facets of dimension $d$. Also, a.a.s. $X$ is an $(\epsilon/4)$-net. Suppose that these two properties are satisfied. Take a maximal simplex $\sigma=[x_0,x_1,\dots,x_d]\in\Del{X}$. 
	We will prove there exists a point $z\in X\setminus\{x_0,\dots, x_d\}$ such that $x_i\sim z$ in $\distgraph[X,\epsilon]$ for all $i=0,\dots, d$. Hence $\sigma\subset N(z)$ and so $\sigma\in\Ne{\distgraph[X,\epsilon]}$ as desired. We will do this by considering two cases:
	
	\begin{enumerate}
		\item \textbf{(When $0<\alpha$)}. For $\epsilon$ small enough, the geodesic Delaunay Triangulation produces the same simplicial complex. Since $\sigma$ is in the Delaunay Triangulation, there is a geodesic ball containing $\sigma$ in its boundary, let $y$ be its center on $\S^d$. Thus $\dist{x_i,y}\leq\epsilon/4$, since all the $x_i$'s are equidistant to $y$ and are its closest points in $X$, which is an $(\epsilon/4)$-net. 
		
		Let now $\hat{y}\in \S^d$ be such that $\dist{y,\hat{y}}=\alpha$. Since $X$ is a net, there must exist $z\in X$ such that $\dist{\hat{y},z}\leq \epsilon/4$. Applying the triangle inequality twice, we get  $\alpha-\epsilon\leq\dist{x_i,z}\leq\alpha+\epsilon$ for all the $x_i$'s. Moreover, since $\epsilon\to0$ we may assume it is small enough so that $\alpha-\epsilon>0$ and hence $\dist{x_i,z}\geq\alpha-\epsilon>0$, guaranteeing $z\neq x_i$ for all $i=0,\dots, d$. 
		
		\item \textbf{(When $\alpha=0$)}. Since $\Del{X}\approx\S^d$, the 1-skeleton $\Del{X}^{(1)}$ must be connected. In particular, there must exist another vertex $z\in X\setminus\{x_0, \dots, x_d\}$ such that $\{z,x_i\}\in\Del{X}^{(1)}$ for some $0\leq i\leq d$, since otherwise $\sigma^{(1)}\subsetneq\Del{X}^{(1)}$ would be a connected component. Applying \cref{lemma:diam_simplex_delaunay} to the simplices $\{x_i,z\}$ and $\{x_0,x_1,\dots,x_d\}$, we get $\dist{x_i,z}\leq\epsilon/2$ and $\dist{x_i,x_j}\leq\epsilon/2$, hence $\dist{z,x_j}\leq\epsilon$ for all $j=0,\dots,d$ and so $x_j\sim z$ in $\distgraph[X,\epsilon][0]$ as wanted.\qedhere 
	\end{enumerate}
	
\end{proof}

\begin{lemma}\label{lemma:conv_hull_misses_zero}
	Fix a dimension $d$ and $0\leq\alpha\leq\pi$, $\alpha\neq\pi/2$. Suppose $\epsilon$ is small enough so that either $\alpha+\epsilon<\pi/2$ or $\alpha-\epsilon>\pi/2$. Let $w\in\S^d$ and let $Y\subset N(w)$ be a finite set of neighbors of $w$ in the graph $\distgraph$. Then, $0\not\in\text{conv}(Y)$, where $\text{conv}(Y)$ is the convex hull of $Y$.
\end{lemma}

\begin{proof}
		Let $w^{-}$ and $w^{+}\in\S^d$ be points such that $\dist[\S^d]{w,w^{\pm}}=\alpha\pm \epsilon$. Notice that all points over $\S^d$ at distance $\alpha\pm\epsilon$ from $w$ lie on $\S^d\cap\Pi_{\pm}(w)$, where $\Pi_{\pm}(w)$ is the $d$-plane given by $$\Pi_{\pm}(w)=\{x\in\R^{d+1}: \langle w,x \rangle = \langle w,w^{\pm}\rangle \}.$$
	Then, all points of $\S^d$ at distance at most $\alpha+\epsilon$ from $w$ lie in the semi-space
	$$H_{+}(w)=\{x\in\R^{d+1}: \langle w,x\rangle\geq \langle w, w^{+}\rangle\},$$
	and similarly, all points of $\S^d$ at distance at least $\alpha-\epsilon$ from $w$ lie in the semi-space
	$$H_{-}(w)=\{x\in\R^{d+1}: \langle w,x\rangle\leq \langle w, w^{-}\rangle\}.$$
	
	We then have two cases:
	\begin{itemize}
		\item If $\alpha<\pi/2$, we supposed $\alpha+\epsilon<\pi/2$, and let $H(w)=H_{+}(w)$. Since $\dist[\S^d]{w,w^+}=\alpha+\epsilon<\pi/2$, we must have $\|w-w^+\|<\sqrt{2}$, and thus
		$$\langle w,w^+\rangle=\frac{2-\|w-w^+\|^2}{2}>0=\langle w,0\rangle.$$
		So $0\not\in H(w)$. 
			
		\item If $\alpha>\pi/2$, we supposed $\alpha-\epsilon>\pi/2$, and let $H(w)=H_{-}(w)$. Since $\dist[\S^d]{w,w^-}=\alpha-\epsilon>\pi/2$, we must have $\|w-w^-\|>\sqrt{2}$, and thus
		$$\langle w,w^-\rangle=\frac{2-\|w-w^-\|^2}{2}<0=\langle w,0\rangle.$$
		So $0\not\in H(w)$.  		
	\end{itemize}
	
	Note that in both cases, all points $x\in\S^d$ such that $\alpha-\epsilon\leq \dist{w,x}\leq \alpha+\epsilon$ are always contained in the semi-space $H(w)$, which doesn't contain 0. Since $H(w)$ is convex, for any subset $Y\subset N(w)\subset H(w)$, so $\text{conv}(Y)\subset H(w)$, and hence $0\not\in\text{conv}(Y)$. 	
\end{proof}
 
\subsection{Proof of \cref{thm:best_conn_neigh_eps_distance_graphs}}

\begin{proof}
	By \cref{lemma:eps_net_regime}, a.a.s.  $X=V(\distgraph[n,\epsilon])$ is an $(\epsilon/4)$-net of $\S^d$, suppose this is the case. Also suppose $\epsilon$ is small enough so that $\Del{X}\approx\S^d$ and either $\alpha+\epsilon<\pi/2$ or $\alpha-\epsilon>\pi/2$. 
	
	Let $K\subset\R^N$ be a geometric realization of $\Ne{\distgraph[X,\epsilon]}$. By \cref{thm:delaunay_subcomplex_neighbor}, $\Del{X}$ is a subcomplex of the neighborhood complex, so in particular there is a subcomplex $L\subset K$ that is a geometric realization of $\Del{X}$. In particular, $L\approx\Del{X}$. 
	
	Since $V(\Ne{\distgraph[X,\epsilon]})=X$, for each point $x\in X$, there is a corresponding point $v_x\in K$. Let $\Phi:K\to\R^{d+1}$ be the piecewise-linear map defined by mapping the vertices of $K$ into their corresponding point on $X\subset\S^d\subset\R^{d+1}$, i.e. $\Phi(v_x)=x$. Note that $\Phi$ restricted to $L$ is a simplicial map $\Phi:L\to\Del{X}$, moreover, it is a simplicial isomorphism, so this restriction is actually a homeomorphism. 	

	Recall all facets of $K$ are of the form $N(w)$ where $w$ is a vertex in $X$. Let $\sigma=N(w)=\{v_0,\dots, v_k\}\in K$ be a facet of $K$. Since $\Phi$ is piece-wise linear, $\Phi(\sigma)$ is contained in the convex hull $\text{conv}\left(v_0,\dots, v_k\right)$. By \cref{lemma:conv_hull_misses_zero}, $0\not\in\text{conv}\left(v_0,\dots, v_k\right)$, so in particular $0\not\in\Phi(\sigma)$. Since this happens for all facets of $K$, we get $0\not\in\Phi(K)$, so actually $\Phi:K\to\R^{d+1}\setminus\{0\}$.
	
	Finally, consider the homotopy equivalence $\Psi:\R^{d+1}\setminus\{0\}\to\S^d$ from \cref{lemma:ambien_homeo_delaunay}, so $\Psi|_{\Del{X}}$ is a homeomorphism. 
	
	Putting all these maps together, we get the following commutative diagram. 
	
\begin{center}
	\begin{tikzcd}[row sep=large,column sep=large]
	K \arrow{r}{\Phi} & \R^{d+1}\setminus\{0\} \arrow[r,"\Psi","\simeq"'] &\S^d\\
	L \arrow[hook]{u}{i} \arrow[r,"\Phi","\approx"']& \Del{X}\arrow[hook]{u}{j} \arrow[r,"\Psi","\approx"'] & \S^d \arrow[u,equal] 
	\end{tikzcd}
\end{center}

Aplying the homology functor $H_d$, we get the following commuting diagram, where the homeomorphisms and the homotopy equivalence become group isomorphisms.

\begin{center}
	\begin{tikzcd}[row sep=large,column sep=large]
		H_d(K) \arrow{r}{\Phi^*} & H_d\left(\R^{d+1}\setminus\{0\}\right) \arrow["\approx"']{r}{\Psi^*} &\Z\\
		\Z \arrow[hook]{u}{i^*} \arrow[r,"\Phi^*","\approx"']& \Z \arrow[hook]{u}{j^*} \arrow[r,"\Psi^*","\approx"'] & \Z \arrow[u,equal] 
	\end{tikzcd}
\end{center}

Thus, by following two different paths from the lower left corner to the upper right corner, we get that $\Psi^*\circ \Phi^*\circ i^*$ is an automorphism of $\Z$, so in particular, the map $i^*$ must be injective, so we get the injection $\Z\hookrightarrow H_d(K)=H_d(\Ne{\distgraph[X,\epsilon]})$ as required. 
Finally, an application of Hurewicz Theorem \cite[see e.g.][Thm.~4.32]{Hatcher2002} implies that $\conn{\Ne{\distgraph[X,\epsilon]}}\leq d-1$. 
\end{proof}

\section[Case $d=1$]{Case $\bm{d=1}$}\label{sec:case_1d}
In this section we compute the chromatic number of $\distgraph[n,\epsilon][\alpha][1]$, for all values of $\alpha$, a.a.s. (provided $\epsilon\to0$ slow enough). We also compute the connectivity of the neighborhood complex for all $\alpha$. This is the only dimension where we know the complete picture. 

At the start of the chapter we mentioned that many of the bounds known for distance graphs on spheres also give the best known bounds for $\epsilon$-distance graphs. This shouldn't be interpreted as if our work on $\epsilon$-distance graphs is futile, since they actually can be different. We have already discussed that $\distgraph[n,\epsilon][\pi]=\Bor[n,\epsilon]$, so a.a.s. its chromatic number is $d+2$, while clearly $\chi(\distgraph[0][\pi])=2$. But here we will show that there also other values of $\alpha$ for which, $\chi(\distgraph[0][\alpha][1])\neq\chi(\distgraph[\epsilon][\alpha][1])$. 

We first compute the chromatic number of distance graphs on $\S^1$. 

\begin{theorem}\label{thm:distance_graphs_s1}
	Let $0<\alpha<\pi$. Then 
	$$\chi\left(\distgraph[0][\alpha][1]\right)=
	\begin{cases}
		3 & \text{ if } \frac{2\pi}{\alpha}\text{ is an odd integer}\\
		2 & \text{ otherwise}
	\end{cases}.$$
\end{theorem}

\begin{proof}\leavevmode
	\begin{enumerate}
		\item Suppose $2\pi/\alpha=2k+1$, so $\alpha=2\pi/(2k+1)$. Then, the regular $(2k+1)$-gon inscribed in $\S^1$ is actually an odd cycle of $\distgraph[0][\alpha][1]$, so $3\leq\chi(\distgraph[0][\alpha][1])$. Moreover, it is clear that we can partition $\S^1$ into disjoint $(2k+1)$-gons, by considering all the polygons \textit{rooted} at a vertex $x$ in the arc $\left[0,\frac{2\pi}{2k+1}\right)$. Thus, coloring each odd cycle independently with 3 colors, gives a proper 3-coloring of the distance graph.
		\item Suppose $\alpha\neq \frac{2\pi}{2k+1}$ for all natural $k$. It's straighforward to see that all odd cycles on $\distgraph[0][\alpha][1]$ must be the vertices of inscribed regular odd-polygons. But there are not odd polygons with side $\alpha\neq\frac{2\pi}{2k+1}$, so the graph has not odd cycles. It is a known fact in graph theory, that all graphs with not odd cycles are bipartite. So $\chi\left(\distgraph[0][\alpha][1]\right)=2$. \qedhere
	\end{enumerate}
\end{proof}

\begin{theorem}
	Let $0<\alpha\leq\pi$, $\alpha\neq\frac{2\pi}{3}$. Then, for $\epsilon$ small enough,	$$\chi(\distgraph[\epsilon][\alpha][1])\leq 3.$$
\end{theorem}

\begin{proof}
	Let $N$ be an integer that we specify later. Let $F_1,\dots, F_N$ be the vertices of a regular $N$-gon in $\S^1$, so $\dist[\S^1]{F_i,F_{i+1}}=\ell$ is constant. Denote by $A_i=[F_i,F_{i+1})\subset\S^1$, the circular arc between $F_i$ and $F_{i+1}$, containing $F_i$ but not $F_{i+1}$.  Note $A_i, \dots, A_N$ are a partition of $\S^1$.
	Now we color $\S^1$ coloring each $x\in A_i$ with color \mbox{$c(x)=i\pmod 3$}.
	
	To prove that this is indeed a proper coloring we need to specify $N$, we do this in three different cases:
	
	\begin{enumerate}
		\item (\textbf{If} $\frac{2\pi}{3}<\alpha\leq\pi$). 	Let $N=3$. So each color is assigned to exactly one arc $A_1,A_2,A_3$. Note $\diam{A_i}=\frac{2\pi}{N}=\frac{2\pi}{3}<\alpha-\epsilon$ provided $\epsilon$ is small enough. No edge can be formed between vertices of the same color, so the coloring is proper. 
		
		\item (\textbf{If} $0<\alpha<\frac{2\pi}{3}$ and $\alpha\neq\frac{2\pi}{3\cdot 2^m}$ for all $m\in \Z$).	
		There exits $m\geq1$ such that $\frac{2\pi}{3\cdot 2^m}<\alpha<\frac{2\pi}{3\cdot 2^{m-1}}<\pi$. Suppose $\epsilon$ is small enough so that		
		$$\frac{2\pi}{3\cdot 2^m}<\alpha-\epsilon<\alpha+\epsilon<\frac{2\pi}{3\cdot 2^{m-1}}.$$		
		Let $N=3\cdot2^m$. Then $\diam{A_i}=\frac{2\pi}{N}<\alpha-\epsilon$, so again no edge exists between vertices in the same arc. Moreover, since $N$ is a multiple of 3, the closest arcs to $A_i$ of the same color must be $A_{i-3}$ and $A_{i+3}$ (indices modulo $N$). But it is clear that%
		$$\dist{A_i, A_{i+3}}=\dist{F_{i+1},F_{i+3}}=\frac{4\pi}{N}>\alpha+\epsilon$$
		So no edges can be formed between vertices of the same color. 
		
		\item (\textbf{If} $\alpha=\frac{2\pi}{3\cdot 2^m}$ for some $m\geq 1$).		
		Then $\frac{2\pi}{3(2^m+1)}<\alpha<\frac{4\pi}{3(2^m+1)}\leq\pi$. So let $N=3(2^m+1)$, and suppose $\epsilon$ is small enough so that		
		$$\frac{2\pi}{3(2^m+1)}<\alpha-\epsilon<\alpha+\epsilon<\frac{4\pi}{3(2^m+1)}$$
		Repeating the argument of Case (2), shows that $c$ is a proper 3 coloring. \qedhere
	\end{enumerate}
	
\end{proof}

\begin{coro}
	Let $0<\alpha\leq\pi$, $\alpha\neq\frac{2\pi}{3}$. Then, for $\epsilon$ small enough
	$$\chi(\distgraph[n,\epsilon][\alpha][1])\leq 3.$$
\end{coro}
\begin{proof}
	Since $\distgraph[n,\epsilon][\alpha][1]\subset \distgraph[\epsilon][\alpha][1]$ then $\chi(\distgraph[n,\epsilon][\alpha][1])\leq\chi(\distgraph[\epsilon][\alpha][1])\leq 3$.
\end{proof}

\begin{coro}\label{coro:chi_dim1_equal_3}
	If $0<\alpha\leq \pi$, $\alpha\neq\frac{2\pi}{3}$, and $\epsilon\geq \C_1(\log n/n)$, with $\epsilon\to 0$. Then a.a.s. $$\chi\left(\distgraph[n,\epsilon][\alpha][1]\right)=3.$$
\end{coro}

\begin{proof}
	Let $x_0\in\S^1$, there exists $x_1$ such that $\dist[\S^1]{x_0,x_1}=\alpha$, so $\{x_0,x_1\}$ is a 2-clique of $\distgraph[0][\alpha][1]$. By \cref{lemma:clique_lower_bound}, $\chi(\distgraph[n,\epsilon][\alpha][1])\geq3$, so, by the previous corollary, the result follows. 
\end{proof}

\begin{theorem}\label{thm:conn_neigh_dim1}
	If $0\leq\alpha\leq \pi$ and $\epsilon\geq \C_1(\log n/n)$, with $\epsilon\to0$, then a.a.s. $\Ne{\distgraph[n,\epsilon][1]}$ is connected
\end{theorem}

\begin{proof}
	By \cref{lemma:eps_net_regime}, a.a.s. $X=V(\distgraph[n,\epsilon][\alpha][1])$ is an $(\epsilon/4)$-net. Suppose this is the case. Denote $K=\Ne{\distgraph[X,\epsilon][\alpha][1]}$.
	
	For any given vertex $x_0\in V(K)=X$, there exists a unique closest vertex $x_1$ to the \textit{right} of $x_0$ (i.e. the clockwise arc $[x_0,x_1]$ intersects $X$ only on $x_0$ and $x_1$). Note that if we prove $\{x_0,x_1\}$ is an edge in $K$, we automatically get that $K$ must be connected. To do this, consider $y$ the midpoint of the clockwise arc $[x_0,x_1]$ and let $y'\in\S^1$ be a point such that $\dist{y,y'}=\alpha$. Since $X$ is an $(\epsilon/4)$-net, we have $\dist{x_0,y}=\dist{x_1,y}\leq \epsilon/4$ and there exists another vertex $x'\in X$ with $\dist{x',y'}\leq\epsilon/4$. Then, for $i=0,1$ we have	\begin{align*}
		\dist{x_i,x'}&\leq \dist{x_i,y}+\dist{y,y'}+\dist{y',x'}\leq \frac{\epsilon}{4} +\alpha +\frac{\epsilon}{4}\leq\alpha+\epsilon\ \text{and}\\
		\dist{x_i,x'}&\geq \dist{y,y'}-\dist{x_i,y}-\dist{y',x'}\geq\alpha-\frac{\epsilon}{4}-\frac{\epsilon}{4}\geq\alpha-\epsilon.
	\end{align*}
	So $\{x_0,x_1\}\subset N(x')$, so $\{x_0,x_1\}\in K$ as needed. 
\end{proof}

Finally, we summarize all our results for the 1-dimensional random $\epsilon$-distance graphs in the following theorem. 

\begin{theorem}\label{thm:1_dim_eps_dist_graphs}
	If $\epsilon\geq\C_1(\log n/n)$, and $\epsilon\to 0$, then a.a.s. the following hold:
	\begin{enumerate}
		\item If $0<\alpha\leq\pi$, then $\conn{\Ne{\distgraph[n,\epsilon][\alpha][1]}}=0.$
		\item If $0<\alpha\leq\pi$, $\alpha\neq\frac{2\pi}{3}$, then $\chi(\distgraph[n,\epsilon][\alpha][1])=3$.
		\item If $\alpha=\frac{2\pi}{3}$, then $\chi\left(\distgraph[n,\epsilon][2\pi/3][1]\right)=4$. 
	\end{enumerate}
\end{theorem}

\begin{proof}
	We may assume $X=V(\distgraph[n,\epsilon][\alpha][1])$ is already an $(\epsilon/4)$-net. 
	\begin{enumerate}
		\item \cref{thm:conn_neigh_dim1} implies $\conn{\Ne{\distgraph[n,\epsilon][\alpha][1]}}\geq 0$, while \cref{thm:best_conn_neigh_eps_distance_graphs} implies $\conn{\Ne{\distgraph[n,\epsilon][\alpha][1]}}\leq 0$. 
		\item This is \cref{coro:chi_dim1_equal_3}.
		\item Since $\alpha=\frac{2\pi}{3}=\ell_1$ is the distance between vertices of the inscribed equilateral triangle, \cref{lemma:clique_lower_bound} gives $4\leq \chi\left(\distgraph[n,\epsilon][2\pi/3][1]\right)$. By coloring the arcs between the vertices of an inscribed square with 4 different colors, we get a proper 4-coloring, so the result follows. \qedhere
	\end{enumerate}
\end{proof}

Let us point out at an interesting insight from this result. As we mentioned in the introduction, our interest in $\epsilon$-distance graphs on spheres is as a generalization to the random Borsuk graphs studied in \cite{Kahle-Martinez2020}. A different generalization of Borsuk graphs arises when we consider the antipodal map on $\S^d$ as free action of the group $\Z_2$, giving the $G$-Borsuk graphs $\Ggraph{G}{X,\epsilon}$ studied in \cite{Martinez2021}. For those, topological tools seem to provide efficient lower bounds to the chromatic number, by computing the $G$-index of Hom-complexes, which are a generalization to the neighborhood complex introduced by Babson--Kozlov \cite{BabsonKozlov2003,BabsonKozlov2006}. 
Note that when $\alpha=\frac{2\pi}{3}$ the graph  $\distgraph[\epsilon][2\pi/3][1]$ is also the $\Z_3$-Borsuk graph $\Ggraph{\Z_3}{\S^1,\epsilon}$, where $\Z_3=\{\1,\nu,\nu^2\}$ acts on $\S^1$ via rotations of an angle $2\pi/3$. Thus, part (1) of \cref{thm:1_dim_eps_dist_graphs} shows that the connectivity of the neighborhood complex is not a tight lower bound for this $\Z_3$-Borsuk graph. This suggests that the neighborhood complex might not be sufficient to lower bound the chromatic number of other $G$-Borsuk graphs, and so the approach taken in \cite[Section~5]{Martinez2021} was necessary. 

\section[Case $d=2$]{Case $\bm{d=2}$}
	
\begin{theorem}\label{thm:lower_conn_dim2}
If $0<\alpha\leq\pi$, $\epsilon\geq\C_2(\log n/n)^{1/2}$ and $\epsilon\to0$. Then, a.a.s. $\conn{\Ne{\distgraph[n,\epsilon][\alpha][2]}}\geq1$. If moreover $\alpha\neq\pi/2$, then a.a.s. $$\conn{\Ne{\distgraph[n,\epsilon][\alpha][2]}}=1.$$ 
\end{theorem}

\begin{proof}
	By \cref{lemma:eps_net_regime}, we may suppose $X=V\left(\distgraph[n,\epsilon][\alpha][2]\right)$ is an $(\epsilon/4)$-net of $\S^d$. By \cref{thm:delaunay_homeomorphic_sphere}, we may also assume $\Del{X}\approx\S^d$, and that both the geodesic and Euclidean Delaunay Triangulations coincide. Denote $G=\distgraph[X,\epsilon][\alpha][2]$. To make our exposition cleaner, we will prove the theorem by stating and using several claims that we prove later. 
	
	Let $w\in X$ be any vertex. Define the set $$Z(w)=\{x\in\S^2: \alpha-\epsilon\leq\dist{x,w}\leq\alpha+\epsilon\}.$$ Let $K(w)\subset\Del{X}$ be the subcomplex given by all the simplices $\sigma\in\Del{X}$ such that all its vertices are contained in $Z(w)$, i.e. such that $V(\sigma)\subset\Z(w)$. We will have the claim.
	\begin{claim}[1]
		The 1-skeleton of $K(w)$ is connected. 
	\end{claim}

	Let now $\{x,y\}\in\Ne{G}$ be any edge in the neighborhood complex. Thus, there must exist some vertex $w\in X$ such that $\{x,y\}\subset N(w)$. Since $N(w)\subset Z(w)$, we must have $x,y\in K(w)$. By the Claim 1, there must be a path $x=x_1, x_2, \dots, x_m=y$, such that $\{x_i,x_{i+1}\}\in K(w)$. This in particular means that all $x_i\in N(w)$, and also that all $\{x_i,x_{i+1}\}\in\Del{X}$. Thus, we will have the following claim. 
	
	\begin{claim}[2]
		The edge $\{x,y\}$ is homotopic to the 1-chain		
\begin{center}
			$\{x_1, x_2\}+\{x_2,x_3\}+\dots+\{x_{m-1},y\}$ in $\Ne{G}$.
\end{center} 
	\end{claim}

	Finally, take any loop $\gamma$ in $\Ne{G}$. Since it must be homotopic to a simplicial loop (by the Simplicial Approximation Theorem \cite[Thm.~2C.1]{Hatcher2002}), we may assume $\gamma=\{y_0,y_1\}+\dots+\{y_{t-1},y_0\}$, where each edge $\{y_i, y_{i+1}\}\in\Ne{G}$.  By Claim 2, each edge $\{y_i, y_{i+1}\}$ is homotopic to a path on $\Del{X}\subset\Ne{G}$. Thus, $\gamma$ itself is homotopic to a loop on $\Del{X}$, which is homeomorphic to $\S^2$. Since $\S^2$ is simply-connected, all loops on $\Del{X}$ must be null-homotopic on $\Del{X}$. And since $\Del{X}\subset\Ne{G}$, $\gamma$ must be null-homotopic on $\Ne{G}$. Therefore, $\Ne{G}$ is simply connected, and thus $\conn{\Ne{G}}\geq 1$. When $\alpha\neq\pi/2$, by \cref{thm:best_conn_neigh_eps_distance_graphs} we get $\conn{\Ne{G}}=1$.
\end{proof}

\subsection{Proving the claims}\label{sec:proof_thm_lower_conn_dim2}
We now prove the claims in the proof to \cref{thm:lower_conn_dim2}.  

\begin{proof}[Proof of Claim (1)]
	For a cleaner proof, we may assume $w=N_0=(0,0,1)$ is the north pole of the sphere. Thus, the set $Z(w)$ is the section of $\S^2$ constrained by the two horizontal planes $\Pi_{-}: z=h^{-}$ and $\Pi_{+}: z=h^{+}$, for heights $h^{\pm}$ such that $\dist[\S^2]{N_0,\Pi_{\pm}}=\alpha\pm\epsilon$. Similarly, let $\Pi_0$ be the plane at height $h_0$,  exactly at distance $\alpha$ from $N_0$. We will simply write $Z=Z(N_0)$ and $K=K(N_0)$. Let $\Psi$ be the homeomorphism from \cref{lemma:ambien_homeo_delaunay}. Note that the restriction $\Psi:K\to\Psi(K)$ is a homeomorphism that maps triangles in $K$ to \textit{geodesic} triangles in $\S^2$. We proceed by distinguishing two cases. 
	\begin{enumerate}
		\item (\textbf{Case } $0<\alpha<\pi$):
		
		Let $\Gamma=H_0\cap\S^2$, so $\Gamma$ is a circle and $Z$ is a small strip of width $2\epsilon$ around $\Gamma$. Note first that $\Gamma\subset\Psi(K)$. Indeed, if $x\in\Gamma$, because the geodesic $\Del{X}$ covers all $\S^2$, there must be a simplex $\sigma\in\Del{X}$ such that $x$ is in the gedesic triangle of $\sigma$. In particular, if $\sigma=[v_1,v_2,v_3]$, as in the proof of \cref{lemma:diam_simplex_delaunay}, we get $\diam{\sigma}\leq\epsilon/2$. So $\dist{x,v_i}\leq\epsilon/2$, and so $v_i\in Z$, so the simplex $\sigma$ is in $K$, and $x\in\Psi(\sigma)\subset\Psi(K)$. 
		
		Now, let $v\in X\cap Z$ be any vertex in $Z$ at a height larger than $h_0$. It is clear that the star of $v$ on the geodesic $\Del{X}$ must be homeomorphic to a closed disk, and in particular, there must be some vertex $v_2$ such that $\{v,v_2\}\in\Del{X}$, such that $v_2$ is at a strictly lower height than $v$. By \cref{lemma:diam_simplex_delaunay}, $\dist{v,v_2}\leq\epsilon/2$, and being at a lower height, we must have $v_2\in Z$. If the height of $v_2$ is lower than $h_0$, we stop, otherwise we continue this process to produce the vertices $v_3, v_4, \dots, v_r$. Each $v_i$ is connected to $v_{i-1}$ in $\Del{X}$, at an strictly lower height than the previous, and all lie inside of $Z$. Moreover, since we stopped, the height of $v_r$ is less than $h_0$. Because $X$ is finite and it covers all of $\S^2$, this process always ends for a finite $r$. Then, the path $v,v_2, v_3,\cdots, v_r$ is contained in $\Psi(K)$ and crosses $\Gamma$. Hence, there is a path from $v$ to $\Gamma$ contained in $\Psi(K)$. 
		
		A similar argument can be applied for all vertices $v\in X\cap Z$ at a lower height than $h_0$. Therefore, for any $u, v\in X\cap Z$, there is a path connecting them contained in $\Psi(K)$. 
		
		\item(\textbf{Case } $\alpha=\pi$):
		
		Let $S_0$ be the south pole, so $Z$ is a circular cap of radius $\epsilon$ around $S_0$.  There must exist a simplex $\sigma\in\Del{X}$ such that $x\in\Psi(\sigma)$, say $\sigma=\{v_1, v_2, v_3\}$. By \cref{lemma:diam_simplex_delaunay}, $\dist{S_0, v_i}\leq\epsilon/2$, so in particular $v_i\in Z$. Thus $\sigma\in K$, and we get $S_0\in\Psi(K)$. 
		
		Let $v\in X\cap Z$ be any vertex and suppose $v\neq S_0$. So $v$ has a positive height. Because the star of $v$ in the geodesic $\Del{X}$ must be a closed disk, there must be a vertex $v_2\in X$ at a lower height than $v$ connected to it, unless $v$ is one of the vertices with the lowest possible height. In the latter, that means $v\in\sigma$ such that $S_0\in\Psi(\sigma)$, so in particular, there is a path from $v$ to $S_0$ contained in $\Psi(\sigma)\subset\Psi(K)$ (since $\diam{\sigma}\leq\epsilon/2)$. Otherwise, $v$ is not at the lowest possible height, then it is connected in $\Del{X}$ to some $v_2$ at a lower height. Clearly $v_2\in Z$, so $\{v,v_2\}\in K$. We may continue this process until we get $v_2, \cdots, v_r$, and as before, all $v_i\in Z$, and there is a path between $v_r$ and $S_0$ contained in $\Psi(K)$. Since there is also a path from $v$ to $v_r$ contained in $\Psi(K)$, then $v$ is connected all the way to $S_0$ in $\Psi(K)$. Since this can be done for any $v\in X\cap Z$, all vertices in $X\cap Z$ are connected by paths contained in $\Psi(K)$. 
	\end{enumerate}
Since $\Psi(K)$ is homeomorphic to $K$, we get that all vertices of $K$ are connected to each other by paths in $K$, and since it is a simplicial complex, it means that it is connected. In particular, its 1-skeleton $K^{(1)}$ is connected. 	
\end{proof}

\begin{proof}[Proof of Claim (2)]
	We inductively prove that $\{x_1,x_k\}$ is homotopic to $\{x_1,x_2\}+\dots+\{x_{k-1},x_{k}\}$. The result is obvious for $k=2$. Suppose it holds for some $2\leq k\leq m-1$. Since all the $x_i\in N(w)$, in particular the triangle $\{x_1, x_k, x_{k+1}\}\subset N(w)$, so $\sigma=[x_1,x_k,x_{k+1}]\in K$, and we get an homotopy $\{x_1,x_{k+1}\}=\{x_1,x_{k}\}+\{x_{k},x_{k+1}\}$ (see \cref{fig:homotopy_triangle}). Therefore, using the induction hypothesis, we get the homotopies:
	$$\{x_1, x_{k+1}\}=\{x_1,x_k\}+\{x_k,x_{k+1}\}=\{x_1,x_2\}+\cdots+\{x_{k-1},x_k\}+\{x_{k},x_{k+1}\}.$$
	Then, taking $k=m$ we get the claim. 
\end{proof}

\begin{figure}
	\centering
	\includegraphics[height=7em]{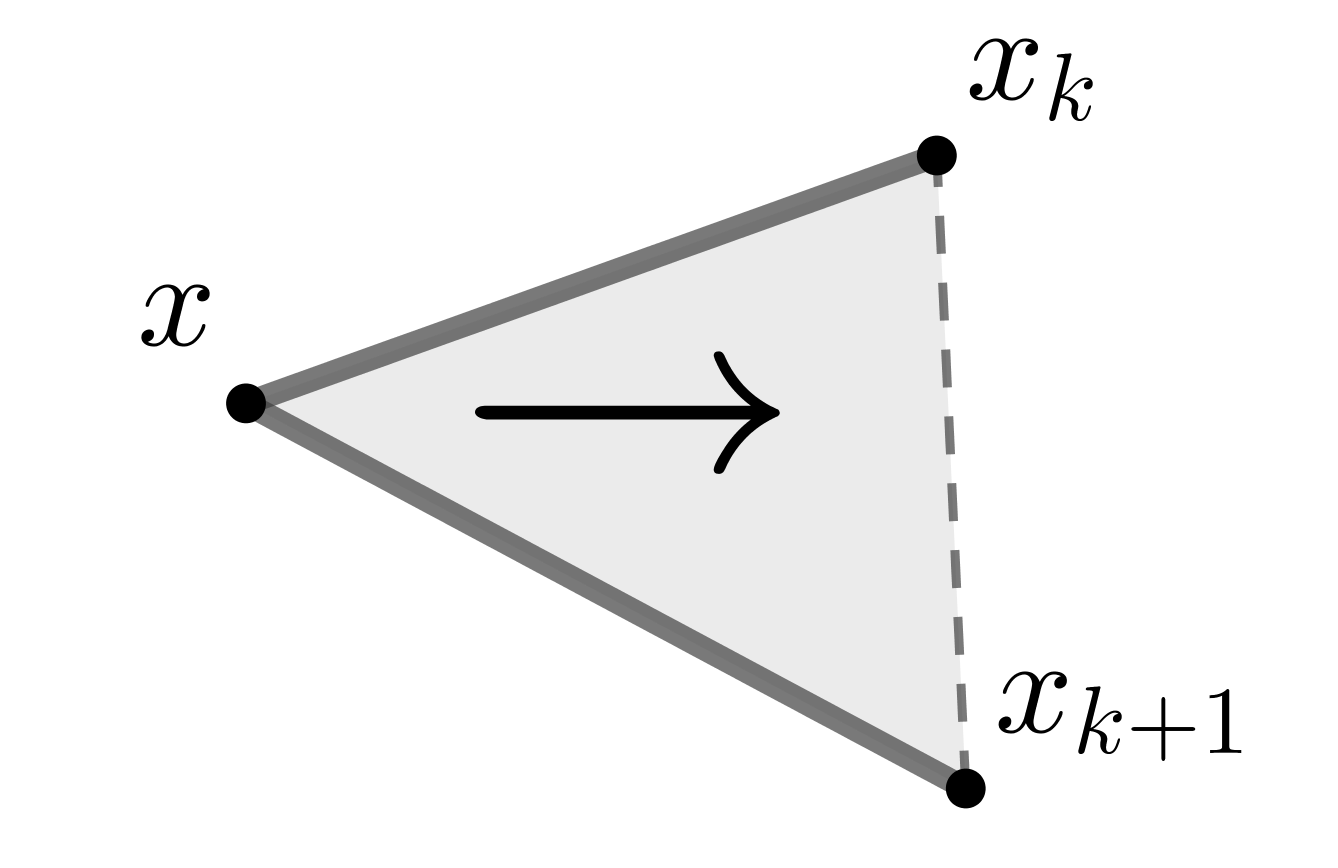}
	\caption{Homotopy between the edges of a triangle.}\label{fig:homotopy_triangle}
\end{figure}

\subsection{Bounds}

We finish this chapter by compiling some results for the upper bound of $\distgraph[0][\alpha][2]$ and comparing them with the lower bounds found, to get intervals for the chromatic number of $\distgraph[n,\epsilon][\alpha][2]$. 

Our best lower bounds come from \cref{thm:lower_conn_dim2} (combined with \cref{thm:conn-neigh-Lovasz}) and \cref{coro:clique_bound_eps_distance_graphs}. Namely, when  $\epsilon\geq\C_2(\log n/n)^{1/2}$, and $\epsilon\to 0$, a.a.s. we get:
\begin{enumerate}
	\item If $0<\alpha\leq \pi$, then $4\leq\chi\left(\distgraph[n,\epsilon][\alpha][2]\right)$. 
	\item If $\alpha=\ell_2$, then $5\leq\chi\left(\distgraph[n,\epsilon][\ell_2][2]\right).$
\end{enumerate}
Where $\ell_2=2\arcsin\left(\frac{\sqrt{6}}{2}\right)\approx1.91$ is the arclength between the vertices of the regular inscribed tetrahedron. 

We now discuss some upper bounds. Let $\lambda_2$ be the diameter of the projection of the faces of the regular tetrahedron onto $\S^2$, so $$\lambda_2=2\arcsin{\sqrt{\frac{1}{2}+\frac{1}{\sqrt{12}}}}\approx2.19.$$
Thus, if we color each of the 4 faces of the tetrahedron with a different color, and project them onto the sphere $\S^2$, we produce a 4-coloring for $\distgraph[\epsilon][\alpha][2]$ for any $\alpha>\lambda_2$, provided $\epsilon$ is small enough (see \cref{fig:S2_4colors}).

For other values of $\alpha$, we can produce upper bounds as we describe below. While these bounds are widely known, we cite \cite{Malen2015} as a reference. 
\begin{enumerate}
	\item If $\alpha>\frac{2\pi}{3}\approx 2.09$, we split $\S^2$ into a circular cap of diameter $2\pi/3$ and 4 triangular stripes, each of diameter $\leq 2\pi/3$ (see \cref{fig:S2_5colors}).
	\item If $\ell^*<\alpha\leq \frac{2\pi}{3}$, where $\ell^*$ is the geodesic diagonal of a face of an inscribed regular dodecahedron, so $\ell^*=2\arcsin\left(\frac{\sqrt{3}}{3}\right)\approx1.23$. Assigning the same color to opposite faces of the regular dodecahedron, we produce a proper 6-coloring, provided $\epsilon$ is small (see \cref{fig:S2_6colors})
\end{enumerate}

\begin{figure}[!htb]
	\captionsetup[subfigure]{justification=centering}
	\centering
	\begin{subfigure}[b]{0.3\textwidth}
		\centering
		\includegraphics[width=\textwidth]{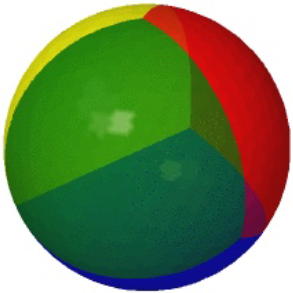}
		\caption{Projecting a tetrahedron (4 colors)\\ \phantom{a} \\ \phantom{a}}\label{fig:S2_4colors}
	\end{subfigure}
	\hfill
	\begin{subfigure}[b]{0.3\textwidth}
		\centering
		\includegraphics[width=\textwidth]{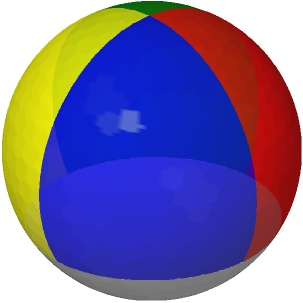}
		\caption{A circular cap and four triangles (5 colors)\\ 
		\phantom{a} \\ \phantom{a}}\label{fig:S2_5colors}
	\end{subfigure}
	\hfill
	\begin{subfigure}[b]{0.3\textwidth}
		\centering
		\includegraphics[width=\textwidth]{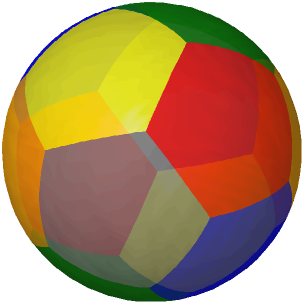}
		\caption{Projecting a dodecahedron, opposite faces given the same color (6 colors)}\label{fig:S2_6colors}
	\end{subfigure}
	\caption{Upper Bounds for $\distgraph[\epsilon][\alpha][2]$ by tessellating the sphere.}
\end{figure}

For smaller values of $\alpha$, other upper bounds can be produced by considering other inscribed polytopes, but we skip this to provide a general bound. In \cite{Coulson2002}, Coulson produces a tessellation of the space $\R^3$ with polytopes, in such a way that it produces proper 15-colorings for $\gendistgraph[\R^3,d]$ for all values of $d$ in a small open interval $(d_0,d_1)$. By suitably re-scaling this coloring, and restricting it to $\S^2$, we can get a proper 15-coloring, for any value $0<\alpha\leq\pi$, provided $\epsilon$ is small. 

\cref{table:bounds_chi_eps_dist_graphs_s2} summarizes the discussed upper and lower bounds for $\distgraph[n,\epsilon][\alpha][2]$, a.a.s. provided $\epsilon\geq\C_2(\log n/n)^{1/2}$ and $\epsilon\to 0$. Note the last row of the table gives the following theorem, where the lower bound was efficiently obtained using the connectivity of the neighborhood complex. 

\begin{theorem}\label{thm:conn_close_to_pi}
If  $\lambda_2<\alpha\leq\pi$, $\epsilon\geq\C_2(\log_n/n)^{1/2}$ and $\epsilon\to0$. Then, a.a.s. $$\chi\left(\distgraph[n,\epsilon][\alpha][2]\right)=4.$$
\end{theorem}

\begin{table}[!htp]
	\centering
	\renewcommand{\arraystretch}{1.5}
	\begin{tabular}{|rcl|c|c|}
		\hline
		&$\pmb{\alpha}$	&                                         & \textbf{Lower Bound} & \textbf{Upper Bound} \\ \hline\hline
		$0<$	&$\alpha$	& $\leq\ell^*\approx 1.23$                           & 4                    & 15                   \\ \hline
		$1.23\approx\ell^*<$ &	 $\alpha$ &$<\ell_2\approx 1.91$               & 4                    & 6                    \\ \hline
		 &$\alpha$			& $=\ell_2\approx1.91$                                 & 5                    & 6                    \\ \hline
		$1.91\approx\ell_2<$ &$\alpha$& $\leq\frac{2\pi}{3}\approx2.09$   & 4                    & 6                    \\ \hline
		$2.09\approx\frac{2\pi}{3}<$ & $\alpha$& $\leq\lambda_2\approx2.19$ & 4                    & 5                    \\ \hline
		$2.19\approx\lambda_2<$ & $\alpha$ & $\leq\pi$                       & 4                    & 4                    \\ \hline
	\end{tabular}
\caption[Bounds on the chromatic number of $\epsilon$-distance graphs of $\S^2$.]{A.a.s. bounds on the chromatic number of $\epsilon$-distance graphs of $\S^2$, provided $\epsilon\to0$ slowly enough.}\label{table:bounds_chi_eps_dist_graphs_s2}
\end{table}

\section{Further work}

\begin{enumerate}[leftmargin=*]
	\item \cref{thm:best_conn_neigh_eps_distance_graphs} asserts that the connectivity of the neighborhood complex can't be higher than $d-1$ for random $\epsilon$-distance graphs on spheres when $\alpha\neq\pi/2$, which in general is a poor lower bound for the chromatic number. We wonder about the case when $\alpha=\pi/2$. Since we don't yet have a conjecture, it would be interesting to illustrate this case by running computer simulations. We observe that this task is not trivial, since we need to choose an $\epsilon$ \textit{small enough} and get an $(\epsilon/4)$-net, so $n$ should be \textit{big} (at least in the order of $\sim10^3$), but then computing the connectivity of $\Ne{\distgraph[n,\epsilon][\alpha][2]}$ (or even homology) requires a lot of computation power and time.
	\item We conjecture that \cref{thm:conn_close_to_pi} can be generalized to higher dimensions. 
	\begin{conj}
	Given $d\geq1$, there exists a constant $\lambda_d<\pi$, such that if $\lambda_d\leq\alpha\leq\pi$, $\epsilon\geq\C_d(\log n/n)^{1/d}$, and $\epsilon\to0$, then a.a.s. $$\conn{\Ne{\distgraph[n,\epsilon]}}=d-1.$$
	And so, a.a.s.
	$$\chi\left(\distgraph[n,\epsilon]\right)=d+2.$$
	\end{conj}

	\item It is also of interest to improve the bounds provided in \cref{table:bounds_chi_eps_dist_graphs_s2}, particularly for $\alpha\geq\ell_2\approx1.91$, since this regime behaves differently than the usual distance graphs on spheres: $\chi(\distgraph[0][\pi][2])=2$ while $\chi(\distgraph[\epsilon][\pi][2])=4$ for any $\epsilon>0$ small enough. 
\end{enumerate}
\printbibliography[title={References}, heading=bibintoc]
\end{document}